\documentclass[a4paper]{article}
\usepackage[T1]{fontenc}
\usepackage[T2A]{fontenc}  
\usepackage[utf8]{inputenc}

\usepackage[english]{babel} 

\usepackage[a4paper, top=3cm, bottom=3cm, left=3cm, right=3cm,       marginparwidth=1.75cm]{geometry}
 \usepackage{appendix}
\usepackage{amsmath}
\usepackage{amssymb}
\usepackage{amsfonts}
\usepackage{mathrsfs}
\usepackage{latexsym}

\usepackage{multirow}
\usepackage{booktabs} 
\usepackage{caption}  
\usepackage{tabularx} 
\newcolumntype{M}{>{\centering\arraybackslash}m{2.5cm}}
\newcolumntype{S}{>{\centering\arraybackslash}m{1.8cm}}
\usepackage[dvips]{graphicx}
\usepackage{xcolor} 
\usepackage[colorinlistoftodos]{todonotes}

\newtheorem{theorem}{\bf Theorem}[section]
\newtheorem{lemma}[theorem]{\bf Lemma}

\newtheorem{remark}[theorem]{\bf Remark}
\newenvironment{proof}{\noindent{\em Proof:}}{\quad \hfill$\Box$\vspace{2ex}}

\numberwithin{equation}{section}

\usepackage{pdflscape}      
\usepackage[figuresright]{rotating} 
\usepackage{float}          

\usepackage[colorlinks=true,
            allcolors=blue]{hyperref}

\title{Sinh regularized Lagrangian nonuniform sampling series}
\author{Haixin Jiang\thanks{Department of Mathematics, Jiujiang University, Jiujiang, 332000, P. R. China. E-mail address: {\it 6070124@jju.edu.cn}.}, \quad Xinyu Chen\thanks{Faculty of Innovation Engineering, Macau university of science and technology, Macau, 999078, P. R. China. E-mail address: {\it 1220001051@student.must.edu.mo}.}, \quad Liang Chen\thanks{ Corresponding Author.   Department of Mathematics, Jiujiang University, Jiujiang, 332000, P. R. China. E-mail address: {\it chenliang3@alumni.sysu.edu.cn}.   ORCID iD: https://orcid.org/0000-0003-3750-1071}} 

\date{}

\begin{document}
\maketitle

\begin{abstract}
Recently, some window functions have been introduced into the nonuniform fast Fourier transform and the regularized Shannon sampling. Inspired by these works, we utilize a sinh-type function to accelerate the convergence of the Lagrangian nonuniform sampling series. Our theoretical error estimates and numerical experiments demonstrate that the sinh regularized nonuniform sampling series achieves a superior convergence rate compared to the fastest existing Gaussian regularized nonuniform sampling series.
\end{abstract}


\section{Introduction}

The celebrated Shannon sampling theorem states that a continuous bandlimited signal can be perfectly reconstructed from its discrete samples. In practice, only a finite number of sample points are available, thus necessitating the consideration of truncation errors for the sampling series. The order of the truncation error (or convergence rate) for the Shannon sampling series is merely $\sqrt{1/N}$, where $N$ denotes the number of samples. To accelerate the convergence rate, oversampling and regularization tricks have been employed \cite{jagerman1966bounds,Wei1997}, achieving exponentially decaying truncation errors, that is,
$$\left|f(x)-\sum_{n=-N}^{N}f(n)\frac{\sin(\pi(x-n))}{\pi(x-n)}g_{N}(x-n)\right|\le \mathcal{O}(\exp(-CN(\pi-\delta))), $$
where $f$ is a function with bandwidth $\delta<\pi$, $g_{N}$ is some regularization function, and  $C$ is some constant.

Compared to uniform sampling (Shannon sampling), nonuniform sampling offers greater flexibility in adapting to complex scenarios and hardware constraints, and has been employed in multi-channel analog-to-digital conversion. The investigation of nonuniform sampling has consistently attracted attention from both the signal processing \cite{selva2008functionally,margolis2008nonuniform,maymon2011sinc,JOHANSSON20152415} and applied mathematics communities \cite{Flornes1999,zayed2001lagrange,strohmer2006fast,annaby2016bounds,adcock2017density,wang2019gaussian,ghosh2023sampling}. The objective of this paper is to improve the reconstruction accuracy of bandlimited functions under the nonuniform sampling through the oversampling and the regularization trick. We consider the regularized Lagrangian nonuniform  sampling series of the following form
$$\sum_{j=-N}^{N} f(\lambda_{j})\frac{F(x)}{F^{'}(\lambda_{j})(x-\lambda_{j})}g_{N}(x-\lambda_{j}),$$
where $F$ is some entire function with zero point set $\{\lambda_{j}\}_{j=-N}^{N}$, and $g_{N}$ is some regularization function.

Recently, \cite{yang2025exponential} extended the Gaussian regularization method (originally from Gaussian regularized Shannon sampling \cite{Wei1998,Wei2000,Qian2003,Tanaka2008,Micchelli2009,Lin2017,chen2019sharp,lin2019optimal,asharabi2022multidimensional,chen2023hyper}) to Lagrangian nonuniform sampling by applying the residue theorem. Before this, \cite{asharabi2023periodic,asharabi2024modification}   applied the Gaussian regularization method to periodic nonuniform sampling series, which is a specific case of Lagrange sampling series. Inspired by \cite{kircheis2024optimal}, we employ the 
following sinh-type regularization function $\varphi_{\beta,m}$ (defined in Eq.(\ref{eq11})) to the Lagrangian nonuniform sampling.
\begin{equation}\label{eq11} 
\varphi_{\beta,m}(x) := 
\begin{cases} 
\frac{1}{\sinh \beta} \sinh \left( \beta \sqrt{1 - \frac{x^2}{m^2}} \right) & : x \in [-m, m], \\ 
0 & : x \in \mathbb{R} \setminus [-m, m].
\end{cases}
\end{equation}
The sinh-type window function has been applied to the regularized Shannon sampling series \cite{Kircheis2022,filbir2023regularized,Kircheis2024} and the nonuniform fast Fourier transform \cite{Plonka2023,kircheis2025nonequispaced}. We prove that  the sinh regularized nonuniform series can achieve a convergence rate nearly twice that of the Gaussian regularized nonuniform sampling series.

\section{Sinh regularized Lagrangian nonuniform sampling series}
For simplicity, we write $ g_1(t)\lesssim g_2(t)$ if $g_1(t)\le C g_2(t)$ for every $t\in \Omega$, where $C$ is some absolute constant and $\Omega$ is the domain of $t$. We write $ g_1(t)\asymp g_2(t)$ if $ g_1(t)\lesssim g_2(t)$ and $ g_2(t)\lesssim g_1(t)$.
For the sake of theoretical analysis, we consider bandlimited functions defined on the Paley-Wiener space. For the bandwidth $\delta>0$, the Paley-Wiener space is defined by 
\[
\mathcal{B}_{\delta}(\mathbb{R}) := \left\{ f \in C(\mathbb{R}) \cap L^2(\mathbb{R}): {\rm supp} \hat{f} \subseteq [-\delta,\delta] \right\},
\]
where $\delta>0$ (bandwidth parameter) determines the frequency support of the Paley-Wiener space  and the Fourier transform $ \hat{f}$ of $f\in L^{1}(\mathbb{R})$  is defined by
$$\hat{f}(w):=\frac{1}{\sqrt{2\pi}}\int_{-\infty }^{ +\infty} f(x) e^{-iwx}dx.$$

For the general regularized sampling series of the function  $f\in \mathcal{B}_{\delta}(\mathbb{R})$, we establish the following result, which reveals that the error is related to the time-frequency concentration of the regularization function.

\begin{lemma}\label{L11}
Let $ N\in \mathbb{N}$, $0<\delta<\pi$, $f\in \mathcal{B}_{\delta}(\mathbb{R})$ and $g_{N}\in C(\mathbb{R})\cap L^{2}(\mathbb{R})$ with $g_{N}(0)=1$ is a regularization function depending on $N$.
Assume $\{\phi_{j}\}_{j=1}^{\infty}$ is a family of basis functions in $\mathcal{B}_{\pi}(\mathbb{R})$ and $\{x_{j}\}_{j=1}^{\infty}\subseteq\mathbb{R}$ is a sample set such that for any $h\in \mathcal{B}_{\pi}(\mathbb{R})$,  the following  interpolation formula holds \begin{equation}\label{eq1}h(x)=\sum_{j=1}^{\infty}h(x_{j})\phi_{j}(x)  \quad\text{for}\quad x\in\mathbb{R}.\end{equation}  
Then, 
$$\sup_{x\in [-1,1]}|f(x) - S_{f,N}(x)|\le E_{1,N}+E_{2,N},$$
where \begin{equation}\label{eqsfn}S_{f,N}(x):=\sum_{j=1}^{N}f(x_{j})\phi_{j}(x)g_{N}(x-x_{j}),\end{equation}

$$E_{1,N}:=\sup_{x\in [-1,1]}\left\{ \frac{2}{\sqrt{2\pi}}\sum_{j=1}^{\infty}|f(x_j)\phi_{j}(x)|\left(\int_{\pi-\delta}^{\infty}|\widehat{g_{N}}(w)|dw+ \int_{-\infty}^{\delta-\pi}|\widehat{g_{N}}(w)|dw\right)\right\}$$ and $$E_{2,N}:=\sup_{x\in [-1,1]}\left\{ \sum_{j=N+1}^{\infty}|f(x_j)\phi_{j}(x)g_{N}(x-x_{j})| \right\}.$$
\end{lemma}

\begin{proof}
For $x\in \mathbb{R}$, let  $$S_{f,\infty}(x):=\sum_{j=1}^{\infty}f(x_{j})\phi_{j}(x)g_{N}(x-x_{j}),$$
and 
$$f(x) - S_{f,N}(x)=:E_{1}(x)+E_{2}(x),$$
where $$E_{1}(x):= f(x)-S_{f,\infty}(x),$$ and $$E_{2}(x):= S_{f,\infty}(x)-S_{f,N}(x).$$
It is obvious that
$\sup_{x\in[-1,1]}|E_{2}(x)|\le E_{2,N}$.
Let $$G_{N}(x)=\frac{1}{\sqrt{2\pi}}\int_{\delta-\pi}^{\pi-\delta}\widehat{g_N}(w)e^{iwx}dw$$ and $$G_{t,N}(x)=G_{N}(t-x).$$ Since ${\rm supp} \widehat{G_{t,N}} \subseteq [\delta-\pi,\pi-\delta] $, $${\rm supp}\widehat{G_{t,N}}*\widehat{f} \subseteq {\rm supp}\widehat{G_{t,N}}+ {\rm supp}\widehat{f} \subseteq[-\pi,\pi].$$ Note that $\widehat{G_{t,N}f}=\widehat{G_{t,N}}*\widehat{f} $,  we have $f(x)G_{N}(t-x)\in  \mathcal{B}_{\pi}(\mathbb{R})$. Using Eq.(\ref{eq1}), we have 
$$f(x)G_{N}(t-x)=\sum_{j=1}^{\infty}f(x_{j})\phi_{j}(x)G_{N}(t-x_{j}).$$
Choosing $t=x$, it follows that
$$f(x)G_{N}(0)=\sum_{j=1}^{\infty}f(x_{j})\phi_{j}(x)G_{N}(x-x_{j})=:s_{f,\infty}(x).$$ Therefore,
$$|E_{1}(x)|=|f(x)-f(x)G_{N}(0)+s_{f,\infty}(x)-S_{f,\infty}(x)|.$$
Noting that $g_{N}(0)=1$ and 
$$g_{N}(x)-G_{N}(x)=\frac{1}{\sqrt{2\pi}}\int_{\pi-\delta}^{\infty}\widehat{g_{N}}(w)e^{iwx}dw+\frac{1}{\sqrt{2\pi}}\int_{-\infty}^{\delta-\pi}\widehat{g_{N}}(w)e^{iwx}dw,$$
we have $\sup_{x\in[-1,1]}|E_{1}(x)|\le E_{1,N}$.

\end{proof}

Next, we give the definition of the sine-type functions \cite{levin1996lectures}. We say an entire function $f$ is a sine-type function with bandwidth $\delta>0$ if 

1. $\inf_{i\neq j }|\lambda_{i}-\lambda_{j}|\ge\lambda>0$, where $\{\lambda_{i}\}_{i=1}^{\infty}=:\Gamma\subseteq \mathbb{R}$ is the zero point set of the function $f$;

2. For every $ z\in\mathbb{C}$ satisfying $ \inf_{i\in \mathbb{N} }|z-\lambda_{i}|\ge \eta$, there exists a positive constant $c_{\eta}$  such that $$|f(z)|\ge c_{\eta}e^{\delta|Im(\mathrm z)|} ,$$ here $Im(\mathrm z)$ representing the imaginary part of $z$, and for every $z\in\mathbb{C}$, there exists  a positive constant  $C_{f,1}$ such that $$|f(z)|\le C_{f,1}e^{\delta|Im(\mathrm z)|} .$$

According to  [17, Theorem 1 in Lecture 23], when $F$ is a sine-type function with bandwidth $\delta$ and $f\in \mathcal{B}_{\delta}(\mathbb{R})$, the Lagrange interpolation formula holds, that is, $$f(x)=\sum_{j\in\mathbb{Z}} \frac{f(\lambda_{j})F(x)}{F^{'}(\lambda_{j})(x-\lambda_{j})}, \quad x\in \mathbb{R}.$$
For a sine-type function $F$ with bandwidth $\delta$, we will use the fact:  \begin{equation}\label{eq13}
\sum_{j\in\mathbb{Z}} \left|\frac{F(x)}{F^{'}(\lambda_{j})(x-\lambda_{j})}\right|^2\le C_{F, \Gamma} \quad \text{for} \quad x\in \mathbb{R},\end{equation}
where $\Gamma:=\{\lambda_{j}\}_{j\in \mathbb{Z}}\subset \mathbb{R}$ is the zero point set of $F$.   Using the definition of the sine-type function, we have 
$$\sup_{x\in \mathbb{R}}|F(x)|\le \sup_{x\in \mathbb{R}}C_{F,1}e^{\delta|Im(\mathrm x)|} =C_{F,1} .$$  Using the Cauchy integral formula for derivatives, we have \begin{equation*}
\begin{split}
\sup_{x\in\{t\in\mathbb{R}||t-\lambda_{j}|\le1\}}\left|\frac{F(x)}{x-\lambda_{j}}\right|&= \sup_{x\in\{t\in\mathbb{R}||t-\lambda_{j}|\le1\}}\left|\frac{F(x)-F(\lambda_{j})}{ x-\lambda_{j}}\right|\lesssim\sup_{\xi\in\{t\in\mathbb{R}||t-\lambda_{j}|\le1\}}\left|F^{'}(\xi)\right|\\&\lesssim \sup_{\xi\in\{t\in\mathbb{R}||t-\lambda_{j}|\le1\}}\int_{z\in\mathbb{C},|z-\xi|=1}\left|\frac{F(z)}{(z-\xi)^2}\right|dz\\&\lesssim \sup_{\xi\in\{t\in\mathbb{R}||t-\lambda_{j}|\le1\}}\int_{z\in\mathbb{C},|z-\xi|=1}\left| F(z) \right|dz\\&\lesssim  \sup_{z\in\mathbb{C},|Im(\mathrm z)|\le 1}|F(z)|
\lesssim \sup_{z\in\mathbb{C},|Im(\mathrm z)|\le 1}C_{F,1}e^{\delta|Im(\mathrm z)|}\\&\lesssim  C_{F,1}e^{\delta }=: C_{F,\delta}.\end{split}\end{equation*} For $s=\{z\in\mathbb{C}:|z-\lambda_{j}|=\frac{\lambda}{2}\}$ (the definition of $\lambda$ see from the definition of the sine-type function). By calculating residues, we have 
$|\frac{1}{F^{'}(\lambda_j)}|\asymp|\int_{s}\frac{1}{F(z)}ds |\lesssim \frac{1}{c_{\lambda}}\lambda$. Using the fact that $\inf_{i\neq j }|\lambda_{i}-\lambda_{j}|\ge\lambda>0$, $|\frac{F(x)}{F^{'}(\lambda_j)}|\lesssim \frac{C_{F,1}}{c_{\lambda}}\lambda$, and $\sup_{x\in\{t\in\mathbb{R}||t-\lambda_{j}|\le1\}}\left|\frac{F(x)}{x-\lambda_{j}}\right|\lesssim C_{F,\delta}$, we have
\begin{equation*}\begin{split}
\sum_{j\in\mathbb{Z}} \left|\frac{F(x)}{F^{'}(\lambda_{j})(x-\lambda_{j})}\right|^{2}&\lesssim \sum_{j\in\mathbb{Z},\lambda_{j}\in (x-1.x+1)} \left|\frac{F(x)}{F^{'}(\lambda_{j})(x-\lambda_{j})}\right|^{2}+\sum_{j\in\mathbb{Z},\lambda_{j}\notin (x-1.x+1)} \left|\frac{F(x)}{F^{'}(\lambda_{j})(x-\lambda_{j})}\right|^{2}\\&\lesssim \frac{\lambda C_{F,\delta}}{c_{\lambda}}\frac{2}{\lambda}+\frac{\lambda C_{F,1}}{c_{\lambda}}\sum_{j\in\mathbb{Z}}\frac{1}{(1+|j\lambda|)^2} =: C_{F, \Gamma} \quad \text{for} \quad x\in \mathbb{R}.\end{split} 
\end{equation*}

Furthermore, from [17,Theorem 1 in Lecture 23], we know that $\left\{\frac{f(\lambda_{j})F(x)}{F^{'}(\lambda_{j})(x-\lambda_{j})} \right\}_{j\in \mathbb{Z}}$ is a Riesz basis   in  $\mathcal{B}_{\delta}(\mathbb{R})$, and then \begin{equation}\label{eq24}(\sum_{j\in \mathbb{Z}}|f(\lambda_{j})|^2)^{1/2}\le C_{F}\|f\|_{L^2},\end{equation} where $C_{F}$ is some constant depending on $F$.

For practical considerations, we restrict our consideration to sampling series with explicit expressions (defined as finite compositions of elementary functions through arithmetic operations).

First, we investigate the periodic nonuniform sampling series \cite{strohmer2006fast,annaby2016bounds}. For any $f\in \mathcal{B}_{\pi}(\mathbb{R})$,  the periodic nonuniform sampling series takes the following form:
\begin{equation}\label{eq33}
f(x) = \sum_{n=-\infty}^{\infty} \sum_{m=1}^{M} f(\tau_{mn}) \psi_{mn}(x)
\end{equation}
where 
\[
\psi_{mn}(x) = \frac{M \prod_{k=1}^{M} \sin((\pi/M)(x - \tau_{kn}))}{\pi(x - \tau_{mn}) \prod_{k=1,k \neq m}^{M} \sin((\pi/M)(t_m - t_k))}
\]
and
\[
\tau_{mn} = t_m + nM, \quad 0 \leq t_1 < t_2 < \cdots < t_M < M, \, n \in \mathbb{Z}.
\]
In fact $$\psi_{mn}(x)=\frac{F_{per}(x)}{F_{per}^{'}(\tau_{mn})(x-\tau_{mn})},$$ where \begin{equation}\label{per}
F_{per}(x)=\prod_{k=1}^{M}\sin(\pi(x-t_{k})/M)\end{equation} is a sine-type function with bandwidth $\pi$.

In addition to periodically nonuniform sampling series, we also consider another type of nonuniform sampling series. To reconstruct the function values on $[-1,1]$, we consider a finite sampling point set $\Lambda:=\{\lambda_{j}\}_{j=-N}^{N}$. Following the works of \cite{selva2008functionally} and \cite{yang2025exponential}, we set 
\begin{equation}\label{genxin}
F_{\Lambda}(z)=\sin(\pi z) \prod_{j=- N }^{N} \frac{z-\lambda_j}{z-j}=:\sin(\pi z)R(z),
\end{equation}  where $\inf_{-N\le j<k\le N}|\lambda_{j}-\lambda_{k}|>0$, $\sup_{|j|\le N}|\lambda_{j}-j|:=L<1$. Let $\bar{\lambda}_{j}:=\lambda_j$ when $|j|\le N$ and $\bar{\lambda}_{j}:=j$ when $|j|>N$. 

 There exists some constant $C_{N,L,\Lambda}>1$ depending on $N$, $L$ and $\Lambda$ such that $\sup_{|z| \ge 4N+4L}|R(z)|\le C_{N,L,\Lambda}$. 
Due to $\sin(\pi z)$ is sine-type function with bandwidth $\pi$ , there exists some constant $c_{\eta}$ depending on $\eta>0$ such that $\sup_{z\in\mathbb{C}}|\sin(\pi z)| \lesssim e^{\pi|Im(\mathrm z)|}$ and  $\inf_{z\in\{y\in\mathbb{C}|\inf_{j\in\mathbb{Z}}|y-\bar{\lambda}_{j}|>\eta\}}|\sin(\pi z)| \ge c_{\eta} e^{\pi|Im(\mathrm z)|}$. Thus, for $z\in \{t\in\mathbb{C}||t|\ge 4N+4L\}$, we have $ |R(z)\sin(\pi z)|\le C_{N,L,\Lambda} e^{\pi|Im(\mathrm z)|} $. Since $\mathbb{Z}$ is zero point set of $\sin(\pi z)$, $\sin(\pi z)R(z)$ is entire function, it follows that
$$ \sup_{|z|<4N+4L}|R(z)\sin(\pi z)|\le \sup_{|z| = 4N+4L}|R(z)\sin(\pi z)|\le C_{N,L,\Lambda} e^{\pi(4N+4L)}  .$$
Therefore, $$ |R(z)\sin(\pi z)|\le C_{N,L,\Lambda} e^{\pi(4N+4L)}e^{\pi|Im(\mathrm z)|} \quad \text{for} \quad z\in\mathbb{C}.$$

Let $V_\eta:=\{y\in\mathbb{C}|\inf_{j\in\mathbb{Z}}|y-\bar{\lambda}_{j}|>\eta\}$, $0<\eta<1$. Note that $|z-\lambda_{j}|\ge\eta $ for $z\in V_\eta$; $|\frac{1}{z-j}|\ge \frac{1}{2N+L+1}$ for $|j|\le N$ and $|z|\le N+L+1$  , thus $$ \inf_{z\in V_\eta,|z|\le N+L+1}|G(z)|\ge \left(\frac{\eta}{2N+L+1}\right)^{2N+1}.$$ 
It follows that,  \begin{equation*}
\begin{split}
\inf_{z\in V_\eta} |G(z)|&\ge\min\left\{\inf_{z\in V_\eta,|z|\le N+L+1}|G(z)|,\inf_{z\in V_\eta,|z|\ge N+L+1}|G(z)|\right\}\\&\ge \min\left\{\inf_{z\in V_\eta,|z|\le N+L+1}\left(\frac{\eta}{2N+L+1}\right)^{2N+1},\inf_{z\in V_\eta,|z|\ge N+L+1}\prod_{j=-N}^{N}\left(1+\frac{j-\lambda_{j}}{z-j}\right)\right\}\\&\ge \min\left\{\left(\frac{\eta}{2N+L+1}\right)^{2N+1}, \left(1-\frac{L}{L+1}\right)^{2N+1}\right\}\\&=:C_{N,\Lambda,\eta}>0\end{split}\end{equation*} 
 Using the fact $\sin(\pi z)\ge C_{\eta}e^{\pi Im(\mathrm z)}$ (since $\sin(\pi z)$ is a sine-type function with bandwidth $\pi$) and the above inequality $\inf_{z\in V_\eta} |G(z)| \ge C_{N,\Lambda,\eta}$, we have \begin{equation*} \begin{split}
 \inf_{z\in V_\eta}|G(z)\sin(\pi z)|\ge C_{\eta}c_{N,\Lambda,\eta}e^{\pi|Im(\mathrm z)}|.\end{split}    \end{equation*} Therefore, Eq.(\ref{genxin}) is a sine-type function with bandwidth $\pi$.  For $j\in \mathbb{Z}$, let \begin{equation}\label{eq12}
Q_{\Lambda,j}(z):=\frac{ F_{\Lambda}(z) }{F_{\Lambda}^{'}(\bar{\lambda}_{j})(z-\bar{\lambda}_{j})}=\frac{R_{j}(z){\rm sinc}(z-j)}{R_{j}(\bar{\lambda}_{j}){\rm sinc}(\bar{\lambda}_{j}-j)},\end{equation}
where $R_{j}(z)=\frac{z-j}{z-\bar{\lambda}_{j}}R(z)$ . Since  Eq.(\ref{genxin}) is a sine-type function with bandwidth $\pi$, we have
\begin{equation}\label{eq44} f(x)= \sum_{j=-\infty}^{\infty}  f(\bar{\lambda}_{j})Q_{\Lambda,j}(x)\quad for \quad x\in \mathbb{R}. \end{equation}

For the nonuniform sampling series (\ref{eq44}) and (\ref{eq33}), using Lemma \ref{L11}, we have following theorem. 
\begin{theorem}
 Let $0<\delta<\pi$,  $N\in \mathbb{N}$ satisfying  $\beta:=(N-1)(\pi-\delta)\ge 1$ .
 For $f\in \mathcal{B}_{\delta}(\mathbb{R})$, define the sinh regularized  non-periodic nonuniform sampling series by $$S_{f,Q,N, \varphi}(x):=\sum_{j=-N}^{N} f(\lambda_{j})Q_{\Lambda,j}(x) \varphi_{\beta,N-1}(x-\lambda_j)$$ 
 and define the sinh regularized periodic nonuniform sampling series by $$S_{f,\psi,N,\varphi}(x):=  \sum_{n=-N}^{N} \sum_{m=1}^{M} f(\tau_{mn}) \psi_{mn}(x) \varphi_{\beta M,(N-1)M} (x-\tau_{mn}) .$$
Then
 \begin{equation}\label{511}    
 \sup_{x\in [-1,1]}|f(x) - S_{f,Q,N,\varphi}(x)|\lesssim C_{\Lambda}\beta\exp(-(N-1)(\pi-\delta))\|f\|_{L^2}  , \end{equation}
and  \begin{equation}\label{611}\sup_{x\in [-1,1]}|f(x) - S_{f,\psi,N,\varphi}(x)|\lesssim C_{F_{per}} M\beta\exp(-(N-1)M(\pi-\delta))\|f\|_{L^2} ,\end{equation}
where  $\Lambda:=\{\lambda_{j}\}_{j=-N}^{N}$, $\inf_{-N\le j<k\le N}|\lambda_{j}-\lambda_{k}|=\lambda>0$, $\sup_{|j|\le N}|\lambda_{j}-j|:=L<1$. $C_{ \Lambda}$ is some constant depending on $\Lambda$, $C_{F_{per}}$ is some constant depending on  $F_{per}$ (see  Eq.(\ref{per})).
\end{theorem}

\begin{proof}
The main idea of the proof is to apply the conclusion of Lemma \ref{L11} (choosing $g_{N}=\varphi_{\beta,N-1}$). Since $\varphi_{\beta,N-1}$ has compact support, it suffices to estimate $E_{1,N}$, and in particular, to estimate $\int_{\pi-\delta}^{\infty}|\widehat{\varphi_{\beta,N-1}}(w)|dw$.

From [16, Eq.(4.8)], we know that
\begin{equation} 
\widehat{\varphi_{\beta,N-1}}(w)=
\frac{(N-1)\sqrt{\pi}}{\sqrt{2}\sinh(\beta)}\cdot \frac{J_{1}(\beta\sqrt{v^2-1})}{(v^2-1)^{1/2}}\quad \text{for}\quad |v|>1 ,\end{equation} 
where $v=\frac{(N-1)w}{\beta}$ and $J_1$ denotes the first order Bessel function \cite{watson1922treatise}. Then 
$$\int_{-\infty}^{\delta-\pi}|\widehat{\varphi_{\beta,N-1}}(w)|dw=\int_{\pi-\delta}^{+\infty}|\widehat{\varphi_{\beta,N-1}}(w)|dw=\frac{\beta\sqrt{\pi}}{\sqrt{2}\sinh(\beta)}\int_{1}^{\infty} \left|\frac{J_{1}(\beta\sqrt{v^2-1})}{(v^2-1)^{1/2}}\right|dv.$$
Since $|J_{1}(x)|\lesssim\left|\sqrt{\frac{2}{\pi x}}\cos(x-\frac{3\pi}{4})\right| \quad for \quad |x|>1$ (see [30, page 199]); $|J_{1}(x)|\lesssim|\frac{x}{4}|\quad \text{for} \quad |x|< 1$ (see [30, page 40]), we have $$|J_{1}(\beta\sqrt{v^2-1})|\lesssim \beta\sqrt{v^2-1}\quad \text{for}\quad  v\in(1,\sqrt{1+1/\beta^2})$$ and $$|J_{1}(\beta\sqrt{v^2-1})|\lesssim 1/\sqrt{\beta\sqrt{v^2-1}}\quad \text{for}\quad  v\in(\sqrt{1+1/\beta^2},\infty),$$  where $\beta\ge 1$. Thus, for $\beta\ge 1$, we have 
\begin{equation*}\begin{split}
\int_{1}^{\infty} \left|\frac{J_{1}(\beta\sqrt{v^2-1})}{(v^2-1)^{1/2}}\right|dv&=  \int_{1}^{\sqrt{1+1/\beta^2}} \left|\frac{J_{1}(\beta\sqrt{v^2-1})}{(v^2-1)^{1/2}}\right|dv+\int_{\sqrt{1+1/\beta^2}}^{\infty} \left|\frac{J_{1}(\beta\sqrt{v^2-1})}{(v^2-1)^{1/2}}\right|dv\\ &\lesssim \int_{1}^{\sqrt{1+1/\beta^2}} \left|\frac{\beta\sqrt{v^2-1}}{(v^2-1)^{1/2}}\right|dv+\int_{\sqrt{1+1/\beta^2}}^{\infty} \left|\frac{1/\sqrt{\beta\sqrt{v^2-1}}}{(v^2-1)^{1/2}}\right|dv\\&\lesssim \beta(\sqrt{1+1/\beta^2}-1)+\int_{\sqrt{1+1/\beta^2}}^{\infty} \left|\frac{1/\sqrt{\beta\sqrt{v^2-1}}}{(v^2-1)^{1/2}}\right|dv
\\&\lesssim  1+\frac{1}{\sqrt{\beta}}\int_{\sqrt{1+1/\beta^2}}^{\infty} \left|\frac{1}{(v^2-1)^{3/4}}\right|dv\\& =1+\frac{1}{\sqrt{\beta}}\int_{1/\beta}^{\infty}\frac{y}{\sqrt{y^{2}+1}}\frac{1}{y^{3/2}}dy\\&\le 1+\frac{1}{\sqrt{\beta}}\int_{1/\beta}^{\infty} \frac{1}{y^{3/2}}dy
\lesssim 1,
\end{split}
\end{equation*}
where we used the estimate $\beta(\sqrt{1+1/\beta^2}-1)= \frac{\beta(1/\beta^2)}{\sqrt{1+1/\beta^2}+1}\le 1/\beta\le 1 \quad \text{for} \quad \beta>1 $.

For $\beta \ge 1$,  $\sinh(\beta)=\frac{\exp(\beta)-\exp(-\beta)}{2}\asymp \exp(\beta) $. Using Eq.(\ref{eq13}), Eq.(\ref{eq24}) and Cauchy-Schwarz inequality, we have
\begin{equation*}\begin{split}\sup_{x\in [-1,1]}(\sum_{j=-\infty}^{\infty}|f(\lambda_j)Q_{\Lambda,j}(x)| )&\left(\int_{\pi-\delta}^{\infty}|\widehat{\varphi_{\beta,N+1}}(w)|dw+ \int_{-\infty}^{\delta-\pi}|\widehat{\varphi_{\beta,N+1}}(w)|dw\right)\\&\lesssim C_{\Lambda} \beta\exp(-(N-1)(\pi-\delta))\|f\|_{L^2}  ,\end{split}\end{equation*} and 
$$\sup_{x\in [-1,1]}(\sum_{j=N+1}^{\infty}|f(\lambda_j)Q_{\Lambda,j}\exp(-\frac{\pi-\delta}{2T-2  }(x-\lambda_{j})^2)|)=0,$$
where $C_{\Lambda}$ is some constant depending on $\Lambda$. Using Lemma \ref{L11} (Choosing $g_{N}=\varphi_{\beta,N-1}$), we proved Eq.(\ref{511}).

Similarly, from Eqs.(\ref{eq13}, \ref{eq24}),  we have \begin{equation*}\begin{split}\sup_{x\in [-1,1]}(\sum_{n=-\infty}^{\infty}\sum_{m=1}^{M}|f(\tau_{mn})\psi_{mn}(x) |)&\left(\int_{\pi-\delta}^{\infty}|\widehat{\varphi_{\beta M,(N-1)M}}(w)|dw+ \int_{-\infty}^{\delta-\pi}|\widehat{\varphi_{\beta M,(N-1)M}}(w)|dw\right)\\&\lesssim C_{F_{per}}\beta M\exp(-NM(\pi-\delta))\|f\|_{L^2},\end{split}\end{equation*}
and 
$$\sup_{x\in [-1,1]}(\sum_{n=N+1}^{\infty}\sum_{m=1}^{M}|f(\tau_{mn})\psi_{mn}(x)\varphi_{\beta M,(N-1)M} (x-\tau_{mn})|)=0.$$
 Eq.(\ref{611}) is proved by using Lemma \ref{L11}.

\end{proof}

\begin{remark}
Eq.(\ref{511}) improves the exponential error term $\exp(-N(\pi -\delta)/2)$ in Gaussian regularized nonuniform sampling \cite{yang2025exponential} to $\exp(-(N-1)(\pi -\delta))$. Here, for a finite sample set $\Lambda$, we only require $L<1$ rather than $L<1/2$ imposed in \cite{yang2025exponential}.
\end{remark}

\begin{remark}
For the sake of simplicity, we do not consider bandlimited functions in the Bernstein space $\mathcal{B}_{\delta}^{\infty}(\mathbb{R})$ (we consider the Paley-Wiener space $\mathcal{B}_{\delta}^{2}(\mathbb{R})$ ), as this would require a slightly more refined estimate of the term $\int_{\pi-\delta}^{\infty}\widehat{g_{N}}(w)e^{iwx}dw$ in Lemma \ref{L11}.  
\end{remark}

\section{Numerical experiments}

In this section, we present numerical experiments to illustrate the performance of the sinh regularized  non-periodic nonuniform sampling series and also  the sinh-type regularized periodic nonuniform sampling series.  We validate our convergence analysis to reconstruct a given bandlimited function $f_{_\delta}\in \mathcal{B}_\delta (\mathbb{R})$ \cite{yang2025exponential}, 
\begin{equation}
\label{funcexample}
f(x) = \frac{1}{\sqrt{\pi(5\delta + \sin (\delta))}} \left( \frac{2 \sin (\delta x)}{x} + \frac{\sin (\delta (x - 1))}{x - 1} \right), \quad x \in \mathbb{R},
\end{equation}
for several bandwidth parameters $\delta \in \{\pi/2, 2\pi/3, 5\pi/6\}$.

 By the definition of the regularized  sampling formula \eqref{eqsfn}, we have
\[
S_{f,N}(x)=\sum_{j=-N}^{N}f(\lambda_j)\phi_{j}(x)g_{N}(x-\lambda_j),\quad x\in[-1,1].
\]

The error shall be estimated by evaluating the given function $f$ and its approximation $S_{f,N}$ at equidistant points $\frac{j}{100}$ with $ j \in \{-100,-99,\cdots,99,100\}$. The corresponding reconstruction errors is measured by
\begin{equation}
\label{errormeasure}  
\max_{-100\leq j \leq 100}\left|f\left(\frac{j}{100}\right)-S_{f,N}\left(\frac{j}{100}\right)\right|.
\end{equation}

First, we consider the sinh regularized  non-periodic nonuniform sampling series, i.e., let $\phi_{j}=Q_{\Lambda,j}$ and $g_{N}=\varphi_{\beta,N-1}$.  Analogous to \cite{yang2025exponential}  we  assume that these sampling points are $\lambda_j = j + \epsilon_j : -N \leq j \leq N$, where  each $\epsilon_j$ is randomly sampled from the uniform distribution on $(-1, 1)$. We average the reconstruction errors over 100 independent random noises $(\epsilon_j)_{j=-N}^{N}$ to reduce the randomness of the results. For  comparison, we use the Gaussian regularized periodic nonuniform sampling series.

The Gaussian regularized non-periodic sampling series  takes the following form 
\begin{equation}
\label{gaussianformula}  
S_{f,Q,N,{\rm Gaussian}}(x)=\sum_{j=-N}^{N}f(\lambda_j)\frac{\text{sinc} (x-j)R_{j}(z)}{\text{sinc}(\lambda_j -j)R_{j}(\lambda_{j})}\exp{\left(-\frac{\pi-\delta}{2N-2}(x-\lambda_j)^2\right)}
\end{equation}
and  the sinh-type regularized non-periodic sampling series is
\begin{equation}
\label{sinhformula}    
S_{f,Q,N,\varphi}(x)=\sum_{j=-N}^{N}f(\lambda_j)\frac{\text{sinc} (x-j)R_{j}(z)}{\text{sinc}(\lambda_j -j)R_{j}(\lambda_{j})}\frac{1}{\sinh \beta} \sinh\left(\beta \sqrt{1 - \frac{(x-\lambda_j)^2}{(N-1)^2}}\right),
\end{equation}
where $\beta=(N-1)(\pi-\delta)$.

Likewise, for the sinh-type regularized periodic nonuniform sampling series, we also consider Gaussian regularized periodic nonuniform sampling series for comparison, which has been established in \cite{wang2019gaussian,asharabi2023periodic,asharabi2024modification}. We average the reconstruction errors over 100 independent random choices $(t_{m})_{m=1}^{M}$ from the uniform distribution on $[0,M)$.

The Gaussian regularized periodic nonuniform sampling series takes the following form
\begin{equation}
\label{periodicgaussianformula}  
\begin{aligned}
S_{f,\psi,N,{\rm Gaussian}}(x)= \sum_{n=-N}^{N} \sum_{m=1}^{M}  f(\tau_{mn}) &
\frac{M \prod_{k=1}^{M} \sin\left((\pi/M)(x - \tau_{kn})\right)}{\pi(x - \tau_{mn}) \prod_{\substack{k=1,k \neq m}}^{M} \sin\left((\pi/M)(t_m - t_k)\right)} \\
&\quad  \cdot  \exp{\left(-\frac{\pi-\delta}{2M(N-1)}(x-\tau_{mn})^2\right)}
\end{aligned}
\end{equation}
and   the sinh-type regularized periodic nonuniform sampling series is
\begin{equation}
\label{periodicsinhformula}  
\begin{aligned}
S_{f,\psi,N,\varphi}(x)= \sum_{n=-N}^{N} \sum_{m=1}^{M}  f(\tau_{mn}) &
\frac{M \prod_{k=1}^{M} \sin\left((\pi/M)(x - \tau_{kn})\right)}{\pi(x - \tau_{mn}) \prod_{\substack{k=1,k \neq m}}^{M} \sin\left((\pi/M)(t_m - t_k)\right)} \\
&\quad  \cdot    \frac{1}{\sinh (M\beta )} \sinh\left(M\beta  \sqrt{1 - \frac{(x-\tau_{mn})^2}{(M(N-1))^2}}\right),
\end{aligned}
\end{equation}
where $\beta=(N-1)(\pi-\delta)$ and $\tau_{mn} = t_m + Mn$. 

\begin{table}[h]
\centering
\caption{Reconstruction errors for $f$ with $\delta = \pi/2$ by  the sinh-type regularized non-periodic sampling formula (\ref{sinhformula}) and  periodic sampling formula (\ref{periodicsinhformula}). }
\label{tab:3pi6}
\begin{tabular}{c c c c c c c c}
\toprule
{N} & {no} & {Gaussian} & {sinh} &{N} & {periodic-no}& {periodic-Gaussian} & {periodic-sinh} \\
\cmidrule(lr){2-4} \cmidrule(lr){6-8}
6  & 1.8510e-01 & 9.2487e-02 & 2.6713e-03 & 2  & 2.6807e-03 & 2.9929e-03 & 7.5429e-05 \\
9  & 7.4302e-02 & 1.9225e-02 & 2.5750e-05 & 3  & 1.2073e-03 & 1.1243e-04 & 3.5914e-07 \\
12 & 1.6440e-02 & 1.5251e-03 & 1.7458e-07 & 4  & 7.7688e-04 & 6.0725e-06 & 2.6398e-09 \\
15 & 2.9575e-02 & 7.3920e-04 & 8.9431e-09 & 5  & 4.8864e-04 & 3.8323e-07 & 1.2914e-11 \\
18 & 5.9010e-03 & 1.7872e-05 & 1.7521e-11 & 6  & 3.6782e-04 & 2.6398e-08 & 6.0840e-14 \\
21 & 7.9550e-03 & 2.0863e-06 & 5.5787e-13 & 7  & 2.6439e-04 & 1.9247e-09 & 2.5535e-15 \\
24 & 2.3567e-03 & 7.0693e-08 & 9.9526e-14 & 8  & 2.1376e-04 & 1.4599e-10 & 4.8850e-15 \\
27 & 1.4045e-02 & 1.3842e-08 & 8.6272e-14 & 9  & 1.6540e-04 & 1.1393e-11 & 4.8850e-15 \\
30 & 5.0550e-03 & 3.2780e-10 & 1.1033e-13 & 10 & 1.3956e-04 & 9.0894e-13 & 4.5519e-15 \\
33 & 2.6273e-02 & 5.6235e-11 & 9.6746e-14 & 11 & 1.1315e-04 & 7.4052e-14 & 5.1070e-15 \\
36 & 3.0467e-03 & 9.7755e-14 & 1.1386e-13 & 12 & 9.8224e-05 & 6.2172e-15 & 5.5511e-15 \\
39 & 3.0713e-03 & 7.3777e-14 & 1.1478e-13 & 13 & 8.2258e-05 & 1.1102e-15 & 5.2180e-15 \\
\bottomrule
\end{tabular}

\end{table}


\begin{table}[h]
\centering
\caption{Reconstruction errors for $f$ with $\delta = 2\pi/3$ by the sinh-type regularized non-periodic sampling formula (\ref{sinhformula}) and  periodic sampling formula (\ref{periodicsinhformula}). }
\label{tab:4pi6}
\begin{tabular}{c c c c c c c c}
\toprule
{N} & {no} & {Gaussian} & {sinh} &{N} & {periodic-no}& {periodic-Gaussian} & {periodic-sinh} \\
\cmidrule(lr){2-4} \cmidrule(lr){6-8}
6  & 2.7989e-02 & 2.0314e-02 & 1.5307e-03 & 2  & 3.9162e-03 & 6.3273e-03 & 1.2798e-03 \\
9  & 1.8443e-02 & 7.7732e-03 & 7.6506e-05 & 3  & 1.9236e-03 & 5.2579e-04 & 8.4274e-06 \\
12 & 7.7737e-03 & 1.5679e-03 & 2.8656e-06 & 4  & 1.1439e-03 & 6.2950e-05 & 1.7286e-07 \\
15 & 1.6494e-02 & 8.4347e-04 & 2.5902e-07 & 5  & 7.5885e-04 & 8.7399e-06 & 1.8372e-09 \\
18 & 1.5889e-02 & 3.6110e-04 & 3.2159e-08 & 6  & 5.4044e-04 & 1.3315e-06 & 4.8339e-11 \\
21 & 8.8869e-03 & 3.9871e-05 & 1.2252e-09 & 7  & 4.0455e-04 & 2.1281e-07 & 4.5149e-12 \\
24 & 4.9444e-03 & 3.4791e-06 & 3.3547e-11 & 8  & 3.1424e-04 & 3.5626e-08 & 2.1327e-13 \\
27 & 3.2907e-03 & 2.5919e-07 & 1.0361e-12 & 9  & 2.5116e-04 & 6.0883e-09 & 7.8826e-15 \\
30 & 3.4577e-02 & 4.7206e-07 & 1.7763e-12 & 10 & 2.0535e-04 & 1.0720e-09 & 3.2196e-15 \\
33 & 1.2174e-03 & 1.2447e-09 & 3.5131e-13 & 11 & 1.7104e-04 & 1.9056e-10 & 3.3307e-15 \\
36 & 1.8003e-03 & 1.1162e-10 & 3.6354e-13 & 12 & 1.4467e-04 & 3.4629e-11 & 4.6629e-15 \\
39 & 7.2980e-04 & 4.9166e-12 & 3.6863e-13 & 13 & 1.2396e-04 & 6.3174e-12 & 4.3299e-15 \\
\bottomrule
\end{tabular}

\end{table}

\begin{table}[h]
\centering
\caption{Reconstruction errors for $f$ with $\delta = 5\pi/6$ by  the sinh-type regularized non-periodic sampling formula (\ref{sinhformula}) and  periodic sampling formula (\ref{periodicsinhformula}). }
\label{tab:5pi6}
\begin{tabular}{c c c c c c c c}
\toprule
{N} & {no} & {Gaussian} & {sinh} &{N} & {periodic-no}& {periodic-Gaussian} & {periodic-sinh} \\
\cmidrule(lr){2-4} \cmidrule(lr){6-8}
6  & 4.1574e-02 & 3.3467e-02 & 8.4079e-03 & 2  & 2.0004e-03 & 8.4344e-03 & 8.0511e-03 \\
9  & 5.3377e-02 & 3.6541e-02 & 4.4741e-03 & 3  & 1.5004e-04 & 1.7230e-03 & 4.1051e-04 \\
12 & 7.0689e-03 & 3.1258e-03 & 1.0807e-04 & 4  & 3.8052e-04 & 4.8913e-04 & 1.2963e-05 \\
15 & 6.2614e-03 & 1.7066e-03 & 3.5911e-05 & 5  & 5.1051e-05 & 1.5351e-04 & 9.3381e-07 \\
18 & 3.2404e-03 & 4.5258e-04 & 3.3881e-06 & 6  & 1.3417e-04 & 5.1311e-05 & 5.9974e-08 \\
21 & 3.6685e-03 & 2.3029e-04 & 6.4206e-07 & 7  & 2.0110e-05 & 1.8634e-05 & 3.5508e-08 \\
24 & 4.0442e-03 & 8.9304e-05 & 2.7462e-07 & 8  & 6.3300e-05 & 7.0514e-06 & 2.2466e-09 \\
27 & 8.7118e-04 & 7.2529e-06 & 8.4722e-09 & 9  & 8.7798e-06 & 2.6534e-06 & 8.7890e-10 \\
30 & 1.1032e-03 & 3.7715e-06 & 6.2906e-09 & 10 & 3.5371e-05 & 1.0116e-06 & 9.5156e-11 \\
33 & 1.3324e-03 & 1.0206e-06 & 8.1806e-10 & 11 & 4.0100e-06 & 4.0379e-07 & 2.2419e-11 \\
36 & 5.7283e-04 & 1.6131e-07 & 1.4600e-10 & 12 & 2.2061e-05 & 1.6415e-07 & 3.8844e-12 \\
39 & 4.4184e-04 & 5.9360e-08 & 1.5924e-11 & 13 & 1.8064e-06 & 6.5601e-08 & 5.2991e-13 \\
\bottomrule
\end{tabular}

\end{table}

We show experimental results for the sinh-type regularized non-periodic and periodic sampling series in Tables \ref{tab:3pi6}, \ref{tab:4pi6}, \ref{tab:5pi6}. We compared the sinh-type regularized non-periodic sampling formula (\ref{sinhformula}) and  periodic sampling formula (\ref{periodicsinhformula}), compared with   the Gaussian regularized non-periodic sampling formula (\ref{gaussianformula}) and  periodic sampling formula (\ref{periodicgaussianformula}). For the non-periodic case, N represents the number of sample points, while for the periodic case, N multiplied by M (which is 3 here) gives the number of sample points. This means that data in the same row represents a comparison between the two cases using the same amount of  sample points. As predicted by our analysis,  the sinh-type regularized sampling formulas exhibit  a faster  convergence rate compared  with those of Gaussian regularized sampling formulas. 

Figures \ref{Sinh-non-periodic} and \ref{Sinh-periodic} show  the logarithm of reconstruction errors  by sinh-type regularizd nonuniform sampling series (including both non-periodic (\ref{sinhformula}) and periodic (\ref{periodicsinhformula}) versions, with $\delta \in \{\pi/2, 2\pi/3, 5\pi/6\}$). These figures compare the reconstruction errors with the main terms of the theoretical errors, expressed as $\exp(-(N-1)(\pi-\delta))$ and $\exp(-(N-1)M(\pi-\delta))$. As shown in Figures \ref{Sinh-non-periodic} and \ref{Sinh-periodic}, the numerical results are consistent with our theoretical error estimates.

 \begin{figure}[h]
   \centering
     \includegraphics[width=0.7\textwidth]{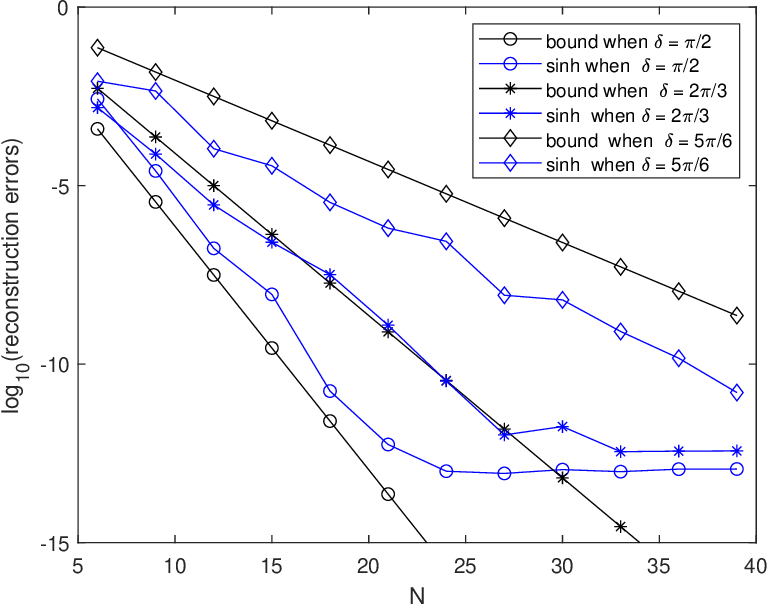}
     \caption{Logarithm of reconstruction errors by the sinh-type regularized non-periodic sampling formula (\ref{sinhformula}). ``bound''$=\log_{10}(\exp(-(N-1)(\pi-\delta))$.}
     \label{Sinh-non-periodic}
 \end{figure}

  \begin{figure}[h]
   \centering
     \includegraphics[width=0.7\textwidth]{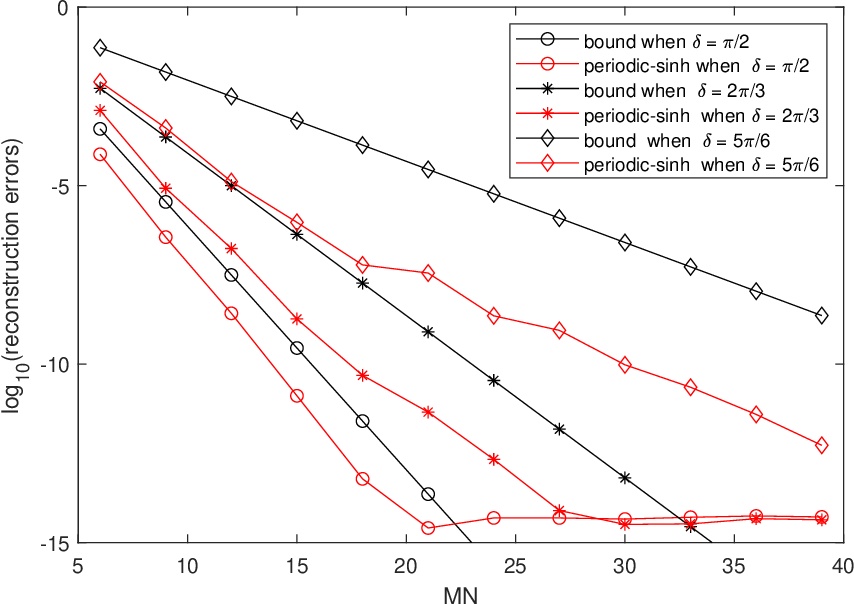}
      \caption{Logarithm of reconstruction errors by the sinh-type regularized periodic sampling formula (\ref{periodicsinhformula}). ``bound''$=\log_{10}(\exp(-(N-1)M(\pi-\delta))$.}
     \label{Sinh-periodic}
 \end{figure}

\section*{Data availability}
No data was used for the research described in the article.

 \section*{Acknowledgment}
 The authors express sincere appreciation to the reviewers for their thorough and patient review of the paper and for pointing out numerous issues. The authors would like to express their gratitude to Yunfei Yang for his assistance in programming and for his guidance in understanding \cite{yang2025exponential}. This work was supported in part
by the National Natural Science Foundation of China under grant 12461019 and  by Jiangxi Provincial Natural Science Foundation(20224BAB201010).

  \section*{Declarations}
\textbf{Competing interests} The authors declare that they have no competing interests.

\bibliographystyle{plain}
\bibliography{sample}
\end{document}